\documentclass[12pt]{article}
\usepackage{amssymb, enumerate}
\usepackage{latexsym}
\usepackage{url}

\usepackage{amsmath,amssymb}
\usepackage[all,cmtip]{xy}

\newtheorem{theorem}{Theorem}[section]
\newtheorem{lemma}[theorem]{Lemma}
\newtheorem{corollary}[theorem]{Corollary}
\newtheorem{proposition}[theorem]{Proposition}

\newtheorem{conjecture}[theorem]{Conjecture}

\newtheorem{definition}[theorem]{Definition}

\newenvironment{proof}{\normalsize {\sc Proof}.}{{\hfill $\Box$%
 \hskip - \parfillskip\bigskip}}

\newcommand{\Syl}{\mathop{\rm Syl}\nolimits}
\newcommand{\Stab}{\mathop{\rm Stab}\nolimits}
\newcommand{\Irr}{\mathop{\rm Irr}\nolimits}

\newcommand{\Hom}{\mathop{\rm Hom}\nolimits}
\newcommand{\Mor}{\mathop{\rm Mor}\nolimits}

\newcommand{\ord}{\mathop{\rm ord}\nolimits}
\newcommand{\Inn}{\mathop{\rm Inn}\nolimits}
\newcommand{\Pic}{\mathop{\rm Pic}\nolimits}

\newcommand{\Aut}{\mathop{\rm Aut}\nolimits}
\newcommand{\Out}{\mathop{\rm Out}\nolimits}

\def\cO{{\mathcal{O}}}
\def\cT{{\mathcal{T}}}
\def\cF{{\mathcal{F}}}

\newcommand{\NN} {\mathbb{N}}
\newcommand{\ZZ} {\mathbb{Z}}
\newcommand{\QQ} {\mathbb{Q}}

\newcommand{\SOF}{\mathop{\rm sf_\mathcal{O}}\nolimits}
\newcommand{\mf}{\mathop{\rm mf}\nolimits}
\newcommand{\mfO}{\mathop{\rm mf_\mathcal{O}}\nolimits}

\newcommand{\nth}{\mathop{\rm th}\nolimits}

\def\a{\alpha}

\def\bigcp{\mathop{\mathchoice 
 {\hbox{\sf\Large\lower 0.1\baselineskip\hbox{Y}}}%
 {\hbox{\sf\large\lower 0.1\baselineskip\hbox{Y}}}%
 {\hbox{\sf\normalsize\lower 0.1\baselineskip\hbox{Y}}}%
 {\hbox{\sf\tiny\lower 0.1\baselineskip\hbox{Y}}}%
}}
\def\bigtimes{\mathop{\mathchoice 
 {\hbox{\sf\Large\lower 0.1\baselineskip\hbox{X}}}%
 {\hbox{\sf\large\lower 0.1\baselineskip\hbox{X}}}%
 {\hbox{\sf\normalsize\lower 0.1\baselineskip\hbox{X}}}%
 {\hbox{\sf\tiny\lower 0.1\baselineskip\hbox{X}}}%
}}

\def\Sym(#1){\mathop{\rm Sym}(#1)}
\def\Sym(#1){S_{#1}}
\def\diag(#1){\mathop{\rm diag}(#1)}

\newenvironment{enumerate*}{%
 \begin{enumerate}%
 }%
 {\end{enumerate}}

\begin{document}

\title{Donovan's conjecture and extensions by the centralizer of a defect group\footnote{This research was supported by the EPSRC (grant no. EP/T004606/1)}}

\author{Charles W. Eaton\footnote{Department of Mathematics, University of Manchester, Manchester M13 9PL. Email: charles.eaton@manchester.ac.uk} and Michael Livesey\footnote{Department of Mathematics, University of Manchester, Manchester M13 9PL. Email: michael.livesey@manchester.ac.uk}}

\date{19th June 2020}
\maketitle


\begin{abstract}
We consider Donovan's conjecture in the context of blocks of groups $G$ with defect group $D$ and normal subgroups $N \lhd G$ such that $G=C_D(D\cap N)N$, extending similar results for blocks with abelian defect groups. As an application we show that Donovan's conjecture holds for blocks with defect groups of the form $Q_8 \times C_{2^n}$ or $Q_8 \times Q_8$ defined over a discrete valuation ring.
\end{abstract}


\section{Introduction}

Let $p$ be a prime and $k:=\overline{\mathbb{F}}_p$. Let $(K,\cO,k)$ be a $p$-modular system, so $\cO$ is a complete discrete valuation ring with residue field $k$. Donovan's conjecture states that for a given finite $p$-group $P$, there are only finitely many Morita equivalence classes amongst blocks of finite groups with defect groups isomorphic to $P$ (this may be stated over $k$ or $\mathcal{O}$). In reducing Donovan's conjecture to quasisimple groups, we must inevitably compare blocks $B$ of finite groups $G$ with those of normal subgroups $N$. The case where $N$ contains a defect group $D$ of $B$ was treated by K\"ulshammer in~\cite{ku95} for $k$-blocks, and by Eisele in~\cite{ei18} for $\mathcal{O}$-blocks. This paper concerns the problematic case of normal subgroups of index a power of $p$, where it suffices to assume $G=ND$. The subcase that $D$ is abelian was first considered, for $k$-blocks and with an additional splitting condition, in~\cite{kk96}. In~\cite{el19} the full $D$ abelian case was treated by introducing strong Frobenius numbers, related to the Morita Frobenius numbers introduced in~\cite{ke04}. The approach taken here also involves strong Frobenius numbers.

The purpose of this paper is to extend the reduction result above to the case that $G=C_D(D\cap N)N$. As an application, we show that Donovan's conjecture with respect to $\cO$ holds when $D \cong Q_8 \times C_{2^n}$ or $Q_8 \times Q_8$ for some $n$. Blocks with defect group $Q_{2^m} \times C_{2^n}$ were studied by Sambale in~\cite{sa13} and the number of irreducible characters and Brauer characters computed. Donovan's conjecture for $\cO$-blocks with defect group $Q_8$ was proved by Eisele in~\cite{ei16}.

In order to make canonical choices of $k$ and $\mathcal{O}$, we choose $k$ to be the algebraic closure of the field $\mathbb{F}_p$ of $p$-elements and $\cO$ to be the ring of Witt vectors for $k$. There is a discussion of this in~\cite{eel19}. When we need to differentiate the versions of Donovan conjecture, the $R$-Donovan conjecture will relate to blocks defined with respect to the ring $R$, where $R$ may be $k$ or $\cO$.

We also consider the following, based on a question of Brauer, which is now often referred to as the weak Donovan conjecture:

\begin{conjecture}
\label{Brauerconj}
Let $P$ be a finite $p$-group. There is $c \in \NN$ such that for all blocks $B$ of finite groups $G$ with defect groups isomorphic to $P$, the entries of the Cartan matrix of $B$ are at most $c$.
\end{conjecture}

In~\cite{ke04} Kessar showed that the $k$-Donovan conjecture is equivalent to showing Conjecture \ref{Brauerconj} and that the Morita Frobenius number (defined in $\S$\ref{sec:strong_frobenius}) of a block is bounded in terms of the order of the defect groups. Variations on the Morita Frobenius number for blocks defined over $\cO$ were given in~\cite{el19}, including the \emph{strong $\cO$-Frobenius number} $\SOF(B)$, and in~\cite{eel19} the analogue of Kessar's result was shown for blocks defined over $\cO$. In~\cite{du04} D\"uvel reduced Conjecture \ref{Brauerconj} to quasisimple groups (although the result in~\cite{du04} is not quite strong enough for our purposes as stated). Our first reduction result concerns the other half of the problem, i.e., bounding the strong $\cO$-Frobenius numbers in terms of the defect groups:

\begin{theorem}
\label{sfreduce:theorem}
Let $G$ be a finite group and $B$ be a block of $\cO G$ with defect group $D$. In addition, let $N \lhd G$ such that $G=C_D(D\cap N)N$ and let $b$ be a block of $\cO N$ covered by $B$. Then $\SOF(B) \leq \SOF(b)$.
\end{theorem}

Combining with Lemma~\ref{Cartan_reduce:lemma}, we obtain the following:

\begin{theorem}
\label{Donovan reduce:theorem}
Let $P$ be a finite $p$-group. In order to verify Donovan's conjecture for $P$ for $\cO$-blocks, it suffices to check it for blocks of finite groups $G$ with defect group $D$ isomorphic to a subgroup of $P$ and no proper normal subgroup $N \lhd G$ such that $G=C_D(D\cap N)N$.
\end{theorem}

After analysing the blocks of quasisimple groups with defect groups $Q_8 \times C_{2^n}$ or $Q_8 \times Q_8$ in Section \ref{sec:classification_application} using~\cite{an20} (in the latter case there are none), we show:

\begin{theorem}
\label{Q8Ctheorem}
There are only finitely many Morita equivalence classes of $\cO$-blocks of finite groups with defect groups isomorphic to $Q_8 \times C_{2^n}$ or $Q_8 \times Q_8$, where $n \geq 0$.
\end{theorem}

The reader might ask what happens for some similar $p$-groups. We note that $Q_8 \times C_{2^n}$ and $Q_8 \times Q_8$ have particularly restricted subgroup structures, meaning that covered blocks of normal subgroups are either amenable to the application of Theorem \ref{sfreduce:theorem} or are easily dealt with by other methods. If, for example, the defect group were $Q_{16} \times C_{2^n}$, then we would have to consider the case of a normal subgroup with defect group $Q_8 \times C_{2^n}$, a case which does arise and to which Theorem \ref{sfreduce:theorem} would not apply. Similarly, $D_8 \times C_2$ is problematic as it contains a normal subgroup $(C_2)^3$. New methods will be needed for these cases. 

The structure of the paper is as follows. In Section \ref{sec:strong_frobenius} we treat strong Frobenius numbers and prove Theorem \ref{sfreduce:theorem}. We prove Theorem \ref{Donovan reduce:theorem} in Section \ref{Donovan_reduce:sec}. In Section \ref{sec:classification_application} we show that there are no blocks of quasisimple groups with defect group $Q_8 \times Q_8$, and few with defect group $Q_8 \times C_{2^n}$ for $n \geq 1$. We give some preliminary reductions and results about blocks with the above defect groups in Section \ref{sec:background}, and the proof that Donovan's conjecture holds for these blocks in Section \ref{sec:Donovan_cases}. 

\section{Strong Frobenius numbers and extensions by the centralizer of a defect group}\label{sec:strong_frobenius}

Throughout this section, let $G$ be a finite group and $B$ a block of $\cO G$ with defect group $D$. We denote by $\Irr(G)$ the set of irreducible characters of $G$ and $\Irr(B)$ the subset of $\Irr(G)$ of irreducible characters
lying in the block $B$. We write $kB$ for the block of $kG$ corresponding to $B$. We denote by $e_B\in \cO G$ the block idempotent for $B$ and by $e_\chi\in \QQ_c G$ the character idempotent for $\chi\in\Irr(G)$, where $\QQ_c$ is the universal cyclotomic extension of $\QQ$. If $A_1$ and $A_2$ are finitely generated $k$-algebras (respectively $\cO$-algebras), we write $A_1\sim_{\Mor}A_2$ if $A_1$ and $A_2$ are Morita equivalent as $k$-algebras (respectively $\cO$-algebras). We quote the following definition from~\cite[Definition 3.2]{el19}.

\begin{definition}
Let $q$ be a, possibly zero or negative, power of $p$. We denote by $-^{(q)}:k\to k$ the field automorphism given by $\lambda\to \lambda^{\frac{1}{q}}$. Let $A$ be a $k$-algebra. We define $A^{(q)}$ to be the $k$-algebra with the same underlying ring structure as $A$ but with a new action of the scalars given by $\lambda.a=\lambda^{(q)}a$, for all $\lambda\in k$ and $a\in A$. For $a\in A$ we define $a^{(q)}$ to be the element of $A$ associated to $a$ through the ring isomorphism between $A$ and $A^{(q)}$. Note that we have $kG\cong kG^{(q)}$ as we can identify $-^{(q)}: kG\to kG$ with the ring isomorphism:
\begin{align*}
-^{(q)}: kG&\to kG\\
\sum_{g\in G}\a_gg&\mapsto \sum_{g\in G}(\a_g)^qg.
\end{align*}

From now on, we identify $(kB)^{(q)}$ with the image of $kB$ under the above isomorphism. We define $B^{(q)}$ to be the unique block of $\cO G$ satisfying $k(B^{(q)})=(kB)^{(q)}$.

By an abuse of notation, we also use $-^{(q)}$ to denote the field automorphism of $\QQ_c$ defined by $\omega_p\omega_{p'}\mapsto \omega_p\omega_{p'}^{\frac{1}{q}}$, for all $p^{\nth}$-power roots of unity $\omega_p$ and $p'^{\nth}$ roots of unity $\omega_{p'}$ and also the ring automorphism
\begin{align*}
-^{(q)}: \QQ_c G&\to \QQ_c G\\
\sum_{g\in G}\a_gg&\mapsto \sum_{g\in G}(\a_g)^{(q^{-1})}g.
\end{align*}

If $\chi\in\Irr(G)$, then we define $\chi^{(q)}\in \Irr(G)$ to be given by $\chi^{(q)}(g) = \chi(g)^{(q^{-1})}$, for all $g\in G$. Note that if $\chi\in\Irr(B)$, then $(e_\chi)^{(q)}=e_{\chi^{(q)}}$ and $\chi^{(q)}\in\Irr(B^{(q)})$.
\end{definition}

For $R\in\{\ZZ,\QQ_c,\cO\}$, we define the $R$-linear map $\zeta_G:RG\to RG$, $g\mapsto g_pg_{p'}^p$, where $g_p$ and $g_{p'}$ are the $p$-part and $p'$-part of $g$ respectively. In general, the relevant ring $R$ should always be clear from the context. The following was proved in~\cite[Proposition 3.5]{el19}.

\begin{proposition}\label{prop:centre}
$\zeta_G$ restricted to $Z(\ZZ G)\subseteq \ZZ G$ induces a $\mathbb{Z}$-algebra automorphism of $Z(\ZZ G)$. Furthermore the algebra automorphism induced on $Z(\QQ_c G)\cong \QQ_c\otimes_{\ZZ} Z(\ZZ G)$ sends $e_\chi$ to $e_{\chi^{(p)}}$, for all $\chi\in \Irr(G)$.
\end{proposition}

We now generalise the above Proposition to deal with normal subgroups of index a power of $p$. In what follows, $A^H$ will denote the set of fixed points in $A$ under the action of $H$, where $A$ is an algebra with an action of a group $H$. In practice, $A$ will always be the group algebra of $G$ or one of its blocks with the natural conjugation action of $H\leq G$. For each $n\in\mathbb{N}$, we use $\omega_n\in \QQ_c$ to denote some fixed primitive $n^{\nth}$ root of unity.

\begin{proposition}\label{prop:zeta_auto}
Let $N\lhd G$ be of index a power of $p$. Then $\zeta_G$ induces a $\ZZ$-algebra automorphism of $(\ZZ G)^N$.
\end{proposition}

\begin{proof}
As noted in the proof of~\cite[Proposition 3.5]{el19}, $\zeta_G$ commutes with conjugation by any $g\in G$. Therefore, since $(\ZZ G)^N$ has a $\ZZ$-basis consisting of the $N$-conjugacy class sums of $G$, $\zeta_G$ maps $(\ZZ G)^N$ to itself. Hence, it is sufficient to prove that $\zeta_G$ induces a $\QQ_c$-algebra automorphism of $(\QQ_c G)^N\cong \QQ_c\otimes_{\ZZ}(\ZZ G)^N$. We do this by showing that $\zeta_G$ induces an isomorphism
\begin{align*}
(\QQ_c G)^Ne_\chi\to(\QQ_c G)^Ne_{\chi^{(p)}},
\end{align*}
for all $\chi\in\Irr(N)$. Note that $\Stab_G(\chi)=\Stab_G(\chi^{(p)})$ and if $g\notin\Stab_G(\chi)$, then
\begin{align*}
C_g^Ne_\chi=e_\chi C_g^Ne_\chi=C_g^Ne_{\chi^g}e_\chi=0,
\end{align*}
where $C_g^N\in\mathbb{Z}G$ is the sum of the elements in the $N$-conjugacy class containing $g$. We may, therefore, assume that $G=\Stab_G(\chi)$. For each $g\in G$ we define
\begin{align*}
e_{\chi,gN}:=\sum_{i=1}^{\ord(gN)}\omega_{\ord(gN)}^i e_{\chi'.\lambda^i},
\end{align*}
where $\chi'$ is an extension of $\chi$ to $\langle g\rangle N$ and $\lambda$ is the linear character of $\langle g\rangle N$ given by
\begin{align*}
\lambda:\langle g\rangle N&\to\QQ_c\\
g^iN&\mapsto\omega_{\ord(gN)}^i.
\end{align*}
If we define the $\QQ_c$-algebra automorphism
\begin{align*}
\tilde{\lambda}:\QQ_c (\langle g\rangle N)&\to \QQ_c (\langle g\rangle N)\\
h&\mapsto \lambda(h)h,
\end{align*}
for all $h\in \langle g\rangle N$, then
\begin{align}\label{algn:lambda_auto}
\begin{split}
\tilde{\lambda}(e_{\chi,gN})=&\sum_{i=1}^{\ord(gN)}\omega_{\ord(gN)}^i \tilde{\lambda}(e_{\chi'.\lambda^i})=\sum_{i=1}^{\ord(gN)}\omega_{\ord(gN)}^i e_{\chi'.\lambda^{i-1}}\\
=&\sum_{i=1}^{\ord(gN)}\omega_{\ord(gN)}^{i+1} e_{\chi'.\lambda^i}=\omega_{\ord(gN)}e_{\chi,gN}.
\end{split}
\end{align}
Therefore, $e_{\chi,gN}\in g(\QQ_cNe_\chi)$. Since we are making a choice of extension $\chi'$ of $\chi$, $e_{\chi,gN}$ is only defined uniquely up to multiplication by some power of $\omega_{\ord(gN)}$. With this in mind, we introduce the notation of $\alpha\approx\beta$, for $\alpha$, $\beta\in(\QQ_c G)^N$, if $\alpha=\mu\beta$ for some $p^{\nth}$-power root of unity $\mu\in\QQ_c$. Note that if when defining $e_{\chi,g^{-1}N}$ we choose the same extension $\chi'$ of $\chi$ as when defining $e_{\chi,gN}$, i.e.,
\begin{align*}
e_{\chi,g^{-1}N}:=\sum_{i=1}^{\ord(gN)}\omega_{\ord(gN)}^i e_{\chi'.\lambda^{-i}}=\sum_{i=1}^{\ord(gN)}\omega_{\ord(gN)}^{-i} e_{\chi'.\lambda^i},
\end{align*}
then
\begin{align*}
e_{\chi,gN}e_{\chi,g^{-1}N}=\sum_{i=1}^{\ord(gN)} e_{\chi'.\lambda^i}=e_\chi.
\end{align*}
Therefore, for any choice of $e_{\chi,g^{-1}N}$, $e_{\chi,gN}e_{\chi,g^{-1}N}\approx e_\chi$ and so we have shown that $(\QQ_c G)^Ne_\chi$ is a crossed product of $G/N$ with $Z(\QQ_c Ne_\chi)=\QQ_c e_\chi$ in the sense of K\"ulshammer~\cite{ku95}. In other words, $(\QQ_c G)^Ne_\chi=\bigoplus_{gN\in G/N}\QQ_c e_{\chi,gN}$. Now
\begin{align}\label{algn:idmpt}
\begin{split}
\zeta_G(e_{\chi,gN})=&\sum_{i=1}^{\ord(gN)}\omega_{\ord(gN)}^i \zeta_G(e_{\chi'.\lambda^i})=\sum_{i=1}^{\ord(gN)}\omega_{\ord(gN)}^i e_{(\chi'.\lambda^i)^{(p)}}\\
=&\sum_{i=1}^{\ord(gN)}\omega_{\ord(gN)}^i e_{\chi'^{(p)}.\lambda^i}\approx e_{\chi^{(p)},gN},
\end{split}
\end{align}
where the second equality follows from Proposition~\ref{prop:centre}. Therefore, $\zeta_G$ induces a $\QQ_c$-vector space isomorphism between $(\QQ_c G)^Ne_\chi$ and $(\QQ_c G)^Ne_{\chi^{(p)}}$.

Before proceeding we note that, for all $g,h\in G$, $(e_{\chi,gN})^{\ord(gN)}=e_\chi$. Also, as noted just after (\ref{algn:lambda_auto}), $e_{\chi,gN}=gxe_\chi$, for some $x\in (\QQ_c Ne_\chi)^\times$ and so $e_{\chi,gN}e_{\chi,hN}e_{\chi,gN}^{-1}=ge_{\chi,hN}g^{-1}\approx e_{\chi,{}^ghN}$. We use both these facts below in (\ref{algn:mult}).

It remains to show that for all $g,h\in G$,
\begin{align*}
\zeta_G(e_{\chi,gN}e_{\chi,hN})=\zeta_G(e_{\chi,gN})\zeta_G(e_{\chi,hN}).
\end{align*}
By (\ref{algn:idmpt}), $\zeta_G(e_{\chi,gN})=e_{\chi,gN}^{(p)}$ and $\zeta_G(e_{\chi,hN})=e_{\chi,hN}^{(p)}$ and so we need only show that $\zeta_G(e_{\chi,gN}e_{\chi,hN})=(e_{\chi,gN}e_{\chi,hN})^{(p)}$. Since $\zeta_G$ fixes coefficients that are $p^{\nth}$-power roots of unity, it suffices in turn to prove that $e_{\chi,gN}e_{\chi,hN}\approx e_{\chi,ghN}$. We prove this last statement via induction on the lowest layer of $gN$, $hN$ or $ghN$ in the upper central series of $G/N$. Note that
\begin{align*}
e_{\chi,gN}e_{\chi,hN}\approx e_{\chi,ghN}\Leftrightarrow e_{\chi,hN}e_{\chi,(gh)^{-1}N}\approx e_{\chi,g^{-1}N}\Leftrightarrow e_{\chi,(gh)^{-1}N}e_{\chi,gN}\approx e_{\chi,h^{-1}N}
\end{align*}
and so we may assume that $hN$ is in the lowest layer. If $h\in N$, then $e_{\chi,hN}=e_\chi$ and so $e_{\chi,gN}e_{\chi,hN}=e_{\chi,ghN}$. Now let $g,h$ be arbitrary elements of $G$ and set $p^n:=\max\{\ord(gN),\ord(hN),\ord(ghN)\}$. Then
\begin{align}\label{algn:mult}
\begin{split}
(e_{\chi,gN}e_{\chi,hN})^{p^n}&\approx (e_{\chi,gN})^{p^n}(e_{\chi,h^{g^{p^n-1}}N}\dots e_{\chi,h^gN} e_{\chi,hN})\approx (e_{\chi,h^{g^{p^n-1}}N}\dots e_{\chi,h^gN})(e_{\chi,h^{-1}N})^{p^n-1}\\
&\approx (e_{\chi,h^{g^{p^n-1}}N}e_{\chi,h^{-1}N})\dots(e_{\chi,h^{gh^{-(p^n-2)}}N}e_{\chi,h^{-1}N})\\
&\approx e_{\chi,[g^{-(p^n-1)},h]N}\dots e_{\chi,[h^{p^n-2}g^{-1},h]N}\approx e_{\chi,[g^{-(p^n-1)},h]\dots[h^{p^n-2}g^{-1},h]N}\\
&=e_{\chi,(gh)^{p^n}N}=e_\chi,
\end{split}
\end{align}
where the fourth and fifth relations follow from the inductive hypothesis. Now, since $e_{\chi,gN}e_{\chi,hN}\in gh(\QQ_c Ne_\chi)\cap(\QQ_c G)^N e_\chi=\QQ_c e_{\chi,ghN}$, we have $e_{\chi,gN}e_{\chi,hN}\approx e_{\chi,ghN}$.
\end{proof}

For the following definitions see~\cite{ke04} and~\cite{el19}.

\begin{definition}
The \textbf{Morita Frobenius number} $\mf(A)$ of a finite dimensional $k$-algebra $A$ is the smallest integer $n$ such that $A\sim_{\Mor} A^{(p^n)}$ as $k$-algebras. The \textbf{$\cO$-Morita Frobenius number} $\mfO(B)$ of $B$ is the smallest integer $n$ such that $B\sim_{\Mor} B^{(p^n)}$ as $\cO$-algebras. The \textbf{strong $\cO$-Frobenius number} $\SOF(B)$ of $B$ is the smallest integer $n$ such that there exists an $\cO$-algebra isomorphism $\phi:B\to B^{(p^n)}$ with the induced bijection of characters given by $\chi\mapsto \chi^{(p^n)}$, for all $\chi\in\Irr(B)$ or, equivalently, that $\phi$ restricted to $Z(B)$ coincides with $\zeta_G^{\circ n}$. Such an isomorphism $\phi$ is called a \textbf{strong Frobenius isomorphism of degree $n$}.
\end{definition}




Before the next lemma we need to give a brief overview of Picard groups of blocks. For a more detailed discussion see~\cite[$\S1$]{bkl18}.

Let $A$ be an $\cO$-algebra. The Picard group $\Pic(A)$ of $A$ consists of isomorphism classes of $A$-$A$-bimodules which induce $\cO$-linear Morita auto-equivalences of $A$. $\Pic(A)$ forms a group with the group multiplication given by tensoring over $A$. For each $\a\in\Aut(A)$ we define ${}_\a A$ to be the $A$-$A$-bimodule with the canonical right action of $A$ and the left action of $A$ given via $\a$. Moreover, ${}_\a A\cong A$ as bimodules if and only if $\a\in\Inn(A)$. In other words, we can view $\Out(A)=\Aut(A)/\Inn(A)$ as a subgroup of $\Pic(A)$.

Now set $A:=iBi$, where $i\in B^D$ is a source idempotent of $B$. We identify $D$ with its image in $A$ and denote by $\Aut_D(A)$ the group of $\cO$-algebra automorphisms of $A$ which fix $D$ pointwise. In addition, we set $\Out_D(A)$ to be the image of $\Aut_D(A)$ in $\Out(A)$, in other words $\Out_D(A)$ is the quotient of $\Aut_D(A)$ by the subgroup of inner automorphisms induced by conjugation with elements in $(A^D)^\times$. We set $\cT(B)$ to be the subgroup of bimodules in $\Pic(B)$ with trivial source, when considered as $\cO(G\times G)$-modules. Through the natural Morita equivalence between $B$ and $A$, we may identify $\Out_D(A)$ with the subgroup of $\cT(B)$ consisting of bimodules with vertex $\Delta D$.

The following Lemma and Theorem~\ref{sfreduce:theorem} were shown for $D$ abelian in~\cite[Theorems 3.15, 3.16]{el19}:

\begin{lemma}\label{lem:graded_units}
Let $N\lhd G$ such that $G=C_D(D\cap N)N$ and $b$ a block of $\cO N$ covered by a block $B$ of $\cO G$ with defect group $D$. Then $e_B=e_b$ and $b$ has defect group $D\cap N$. Viewing $b$ as a subalgebra of $B$, there is a choice of $a_{gN}\in (gb)^\times$ for each $gN\in G/N$ such that $B^N=\bigoplus_{gN\in G/N}a_{gN}Z(b)$.
\end{lemma}

\begin{proof}
By~\cite[Corollary 5.5.6]{nats88}, $B$ is the unique block of $\cO G$ covering $b$ and so $e_B$ is the sum of $e_b$ and its $G$-conjugates. However,~\cite[Theorem 5.5.10(v)]{nats88} gives that $b$ is $G$-stable, so that $e_B=e_b$ and $D \cap N$ is a defect group for $b$. 

Since $B^N$ is $G/N$-graded, it remains to find a unit $a_{gN}$ in the $gN$-graded component of $B^N$, for each $gN\in G/N$. Let $g\in C_D(D\cap N)$ and consider $c_g\in\Aut(B)$ given by conjugation by $g$. Now $c_g$ induces the element
\begin{align*}
M:=\cO_{\Delta c_g}\uparrow^{N\times N}e_b\in\Pic(b),
\end{align*}
where $\Delta c_g=\{(h,c_g(h))|h\in N\}\leq N\times N$ and $\cO_{\Delta c_g}$ is the trivial $\cO(\Delta c_g)$-module. In particular, $M\in\cT(b)$ and so, by~\cite[Theorem 1.1(i)]{bkl18}, $M$ has vertex $\{(h,c_g(h))|h\in D\cap N\}=\Delta(D\cap N)$. Therefore, by the comments preceding the lemma, $iMi\in\Out_{D\cap N}(ibi)$, where $i$ is a source idempotent for $b$. Now by~\cite[14.5, Proposition 14.9]{pu88}, $\Out_{D\cap N}(ibi)$ is a $p'$-subgroup of $\cT(b)$. In particular, $iMi$ has $p'$-order in $\Out(ibi)$ or, equivalently, $M$ has $p'$-order in $\Out(b)$. However, since $g\in D$, $c_g$ has order a power of $p$ meaning $M$ induces the trivial auto-equivalence and $c_g\in \Inn(b)$.

Let $gN$ be a left coset of $N$ in $G$, where we choose coset representative $g\in C_D(D\cap N)$. Set $\a_g\in b^\times$ such that $c_g$ is given by conjugation by $\a_g$. We set $a_{gN}:=g(\a_g)^{-1}$.
\end{proof}

Note that, if we assume the stronger condition $D=(D\cap N)Z(D)$, then $B^N$ can be replaced with $Z(B)$ in Lemma~\ref{lem:graded_units}. The result would then be proved via induction on $|G/N|$, since when $G/N$ is cyclic the $a_{gN}$'s constructed above are in $Z(B)$. We now prove Theorem~\ref{sfreduce:theorem}.
\newline

{\sc Proof of Theorem}~\ref{sfreduce:theorem}.
Set $n:=\SOF(b)$ and let $\phi:b\to b^{(p^n)}$ be a strong $\cO$-Frobenius isomorphism of degree $n$. Then, by Proposition~\ref{prop:zeta_auto}, we can extend $\phi$ to an isomorphism $\tilde{\phi}:B\to B^{(p^n)}$ by sending $a_{gN}$ to $\zeta_G^{\circ n}(a_{gN})$, for each left coset $gN$ of $N$ in $G$ and the $a_{gN}$'s are as in Lemma~\ref{lem:graded_units}. By construction, $\tilde{\phi}$ agrees with $\zeta_G^{\circ n}$ on $B^N$ so certainly it does on $Z(B)$ and we have $\SOF(B)\leq n$.

In the case $D=(D\cap N)Z(D)$, the proof of Theorem~\ref{sfreduce:theorem} can be greatly shortened as we may avoid reference to Proposition~\ref{prop:zeta_auto}. Indeed, given the comments following Lemma~\ref{lem:graded_units}, we only need that $\zeta_G$ induces an $\cO$-algebra automorphism between $Z(B)$ and $Z(B^{(p)})$. The proof then follows in very much the same vein as the abelian defect group case in~\cite[Theorems 3.16]{el19}.\hfill $\Box$


\section{A reduction theorem for Donovan's conjecture}
\label{Donovan_reduce:sec}

The following reduction result for Cartan invariants is presumably well-known, but we provide a proof. Write ${\rm c}(B)$ for the largest entry of the Cartan matrix of a block $B$.

\begin{lemma}
\label{Cartan_reduce:lemma}
Let $G$ be a finite group and $B$ a block of $kG$.
\begin{enumerate}[(i)]
    \item Let $N \lhd G$ have index $p^r$ and suppose $B$ covers a block $b$ of $kN$. Then ${\rm c}(B) \leq p^r {\rm c}(b)$.
    \item Let $Z\leq Z(G)$ be a $2$-group and $\overline{B}$ the corresponding block of $G/Z$. Then ${\rm c}(\overline{B}) \leq {\rm c}(B)$.
\end{enumerate}

\end{lemma}

\begin{proof}
\begin{enumerate}[(i)]
\item Let $B'$ be the unique block of $k\Stab_G(b)$ covering $b$. Then, by~\cite[Theorem C]{ku81}, $B'\sim_{\Mor}B$ and so we may assume that $b$ is $G$-stable and, as noted in the proof of Lemma~\ref{lem:graded_units}, that $e_B=e_b$. Therefore, by Green's Indecomposability Theorem, for every projective, indecomposable $b$-module $M$, the induced module $M\uparrow^G$ is a projective, indecomposable $B$-module. Furthermore, since every projective $B$-module $L$ is a summand of $e_B(L\downarrow_N\uparrow^G)=(e_b L\downarrow_N)\uparrow^G$, in fact every projective, indecomposable $B$-module is isomorphic to $M\uparrow^G$ for some projective, indecomposable $b$-module $M$.

Now fix some projective, indecomposable $b$-module $M$ and consider the composition factors of $M\uparrow^G$. Let $S$ be a simple $B$-module and $T$ a simple $b$-module appearing as a composition factor in $S\downarrow_N$. Certainly the multiplicity of $S$ among the composition factors of $M\uparrow^G$ is at most the multiplicity of $T$ among the composition factors of $(M\uparrow^G)\downarrow_N$. However, $(M\uparrow^G)\downarrow_N$ is the direct sum of $p^r$ projective, indecomposable $b$-modules, the $G$-conjugates of $M$. The claim follows.

\item For a $kG$-module $M$ we define the $k(G/Z)$-module ${}^ZM$ to be the fixed points of $M$ under left multiplication by $Z$. Note that ${}^Z(kG)\cong k(G/Z)$. Therefore, if $M$ is the projective cover of the simple $B$-module $S$, then ${}^ZM$ is the projective cover of the simple $\overline{B}$-module ${}^ZS=S$. In particular, ${\rm c}(\overline{B}) \leq {\rm c}(B)$.
\end{enumerate}
\end{proof}





We now prove Theorem~\ref{Donovan reduce:theorem}.
\newline
\newline
{\sc Proof of Theorem}~\ref{Donovan reduce:theorem}. Fix a finite $p$-group $P$.

Let $\cal{X}$ be the class of blocks $C$ of $\cO H$ for some finite group $H$ with defect group $Q$ isomorphic to a subgroup of $P$ and no normal subgroup $M \lhd H$ such that $H=C_Q(Q \cap M)M$. Suppose that there are only finitely many Morita equivalence classes amongst the members of $\cal{X}$. Then there is a largest Cartan invariant $c$ and a largest strong $\cO$-Frobenius number $s$ amongst blocks in $\cal{X}$. 

Let $G$ be a finite group and $B$ a block of $\cO G$ with defect group $D$ isomorphic to a subgroup of $P$. We claim that ${\rm c}(B)\leq |D|c$ and $\SOF(B)\leq s$. Suppose that ${\rm c}(B)>|D|c$ and that $|G|$ is minimal with respect to these conditions. By the definition of the constant $c$, $B$ is not in $\cal{X}$ and so there is a proper subgroup $N \lhd G$ with $G=C_D(D \cap N)N$. Let $b$ be a block of $\cO N$ covered by $B$. Note that by Lemma~\ref{lem:graded_units} $b$ has defect group $D\cap N$. Then by Lemma~\ref{Cartan_reduce:lemma}(i)
\begin{align*}
{\rm c}(b) \geq [D:D \cap N]^{-1}{\rm c}(B) > [D\cap N]c,
\end{align*}
contradicting the minimality of $G$. A similar argument using Theorem~\ref{sfreduce:theorem} shows the bound on strong $\cO$-Frobenius numbers. We have shown that the Cartan invariants and the strong $\cO$-Frobenius numbers of all blocks with defect group isomorphic to $P$ are bounded, and so the result follows by~\cite[Corollary 3.11]{eel19}.\hfill $\Box$



\section{Blocks of quasisimple groups with defect groups $Q_8 \times C_{2^n}$ and $Q_8 \times Q_8$}
\label{sec:classification_application}

Let $G$ be a block of a finite group $G$ with defect group $D$ and maximal $B$-subpair $(D,b_D)$. Recall that $B$ is \emph{controlled} if for all $B$-subpairs $(Q,b_Q) \leq (D,b_D)$ and $g \in G$ with $(Q,b_Q)^g \leq (D,b_D)$, there are $c \in C_G(Q)$ and $n \in N_G(D,b_D)$ such that $g=cn$. We will use~\cite[Theorem 4.8]{st06} to observe that every block with defect group $Q_8^m \times A$ for $m \geq 0$ and $A$ an abelian $2$-group is controlled, and apply the classification of controlled blocks of quasisimple groups given in~\cite{an20}. To do so we first review some notation.

Let $\mathcal{F}$ be a saturated fusion system on a $p$-group $D$. A subgroup $Q \leq D$ is weakly $\cF$-closed if for any $\phi \in \Hom_\cF(Q,D)$ we have $\phi(Q)=Q$, and $Q$ is strongly $\cF$-closed if for any $P \leq Q$ and any $\phi \in \Hom_\cF(P,D)$ we have $\phi(P) \leq Q$.

The normalizer $N_\cF(D)$ is the fusion subsystem of $\cF$ on $D$ such that for all $P,Q \leq D$, the morphisms $\Hom_{N_\cF(D)}(P,Q)$ are those $\phi \in \Hom_\cF(P,Q)$ such that there is $\bar{\phi} \in \Hom_\cF(D,D)$ extending $\phi$.

A $p$-group $D$ is called \emph{resistant} if $\mathcal{F}=N_{\mathcal{F}}(D)$ whenever $\mathcal{F}$ is a saturated fusion system on $D$.

\begin{proposition}
Let $D=Q_8^m \times A$, where $m \geq 0$ and $A$ is an abelian $2$-group. Then $D$ is resistant. 
\end{proposition}

\begin{proof}
Let $\mathcal{F}$ be a saturated fusion system on $D$. By~\cite[Theorem 4.8]{st06} $\mathcal{F}=N_{\mathcal{F}}(D)$ if and only if there is a central series $D=Q_n \geq Q_{n-1} \geq \cdots \geq Q_1 \geq 1$ with each $Q_i$ weakly $\mathcal{F}$-closed and $D$ strongly $\cF$-closed.

Now let $Q_1=\Omega_1(D)$, the (unique) largest elementary abelian subgroup of $D$, and $Q_2=D$, so $Q_2 \geq Q_1 \geq 1$ forms a central series. Since $Q_1$ is the unique elementary abelian subgroup of $D$ of maximal rank, it must be weakly $\mathcal{F}$-closed. Also $D$ is automatically strongly $\cF$-closed, so the result follows.
\end{proof}

Now let $\mathcal{F}$ be the fusion system on $D$ afforded by $B$ and $(D,b_D)$, sometimes written $\mathcal{F}_{(D,b_D)}(G)$. By Alperin's fusion theorem $B$ is controlled if and only if $N_G(D,b_D)$ controls strong fusion in $D$ (see~\cite[Proposition 4.24]{ab79}). This is equivalent to $\mathcal{F}=N_\mathcal{F}(D)$. We conclude that every block with defect group $D \cong Q_8^m \times A$, where $m \geq 0$ and $A$ is an abelian $2$-group, is controlled.

\begin{proposition}
\label{none_in_quasisimples:prop}
Let $G$ be a finite quasisimple group. Let $m \geq 1$ and $A$ be a finite abelian $2$-group. If $m > 1$, then there is no block of $G$ with defect groups isomorphic to $(Q_8)^m \times A$. If $B$ is a block of $G$ with defect group $Q_8 \times A$ where $A$ is nontrivial, then $G$ is a quotient of a classical group of Lie type not of type $\mathbf{A}$ or ${}^2\mathbf{A}$, defined over a field of order $q$ a power of an odd prime and $B$ corresponds to a non-quasi-isolated block of the corresponding group of Lie type. Furthermore, if $m=1$ and $A \cong C_4$, then $G \cong Sp_{2r}(q)$ for some $r$.
\end{proposition}

\begin{proof}
By the discussion above, all blocks with these defect groups are controlled. The result follows directly from~\cite[Theorem 1.1]{an20} and its proof.
\end{proof}

When we come to consider blocks of arbitrary groups with defect group $Q_8 \times Q_8$, we must deal with the case of blocks $b$ of $Sp_{2r}(q)$ with defect group $Q_8 \times C_4$. We show that there can be no overgroup of $Sp_{2r}(q)$ in $\Aut(Sp_{2r}(q))$ possessing a block covering $b$ with defect group $Q_8 \times Q_8$.

In what follows, for $r\in\mathbb{N}$ and $q$ a power of a prime, we set $I_r \in GL_r(q)$ to be the $r\times r$ identity matrix,
\begin{align*}
J_r:=
\begin{pmatrix}
0 & 0 & \cdots & 0 & 1 \\
0 & 0 & \cdots & 1 & 0 \\
\vdots & \vdots & \ddots & \vdots & \vdots \\
0 & 1 & \cdots & 0 & 0 \\
1 & 0 & \cdots & 0 & 0
\end{pmatrix}\in GL_r(q)
\end{align*}
and
\begin{align*}
\Omega_{2r}:=
\begin{pmatrix}
0 & J_r \\
-J_r & 0
\end{pmatrix}\in GL_{2r}(q).
\end{align*}
We define
\begin{align*}
Sp_{2r}(q)&:=\{x\in GL_{2r}(q)|x\Omega_{2r}x^T=\Omega_{2r}\},\\
CSp_{2r}(q)&:=\{x\in GL_{2r}(q)|x\Omega_{2r}x^T=\lambda \Omega_{2r} \text{ for some }\lambda\in\mathbb{F}_q^\times\}.
\end{align*}

\begin{lemma}
\label{noQ8xQ8ext:Lemma}
Let $H:=Sp_{2r}(q)$, where $4\mid(q-1)$ but $8\nmid(q-1)$ and $b$ a block of $\cO H$ with defect group $P \cong Q_8\times C_4$ labelled by some $s\in H^*=SO_{2r+1}(q)$ such that $m_{X-1}(s)=3$ and other $m_\Gamma(s)\leq 1$. Furthermore, let $H\lhd G$ such that $C_G(H)\leq H$ and $B$ a block of $\cO G$ lying above $b$ with defect group $D$. Then $D\ncong Q_8\times Q_8$.
\end{lemma}

\begin{proof}
We first describe $P\leq H$ (see~\cite{fs89} for a description of defect groups of finite classical groups). We denote by $i\in\mathbb{F}_q^\times$ a primitive $4^{\nth}$ root of unity. Then $P=P_1\times P_2\leq Sp_2(q)\times Sp_{2(r-1)}(q)\leq Sp_{2r}(q)$, where $P_1=\Syl_2(Sp_2(q))\cong Q_8$ is generated by
\begin{align*}
\begin{pmatrix}
i & 0 \\
0 & -i
\end{pmatrix},
\begin{pmatrix}
0 & 1 \\
-1 & 0
\end{pmatrix}
\end{align*}
and $P_2\cong C_4$.

Since $C_G(H)\leq H$, $G/H\hookrightarrow\Out(H)$ and as $8\nmid(q-1)$, $q$ is not a square and so $H$ has no field automorphisms of order $2$. It follows that $\Out(H)$ has a normal Sylow $2$-subgroup of order $2$ generated by the diagonal automorphism induced by
\begin{align*}
\begin{pmatrix}
iI_r & 0 \\
0 & I_r
\end{pmatrix}.
\end{align*}
Furthermore, this automorphism restricted to $Sp_2(q)\times Sp_{2(r-1)}(q)$ is induced by
\begin{align*}
(g_1,g_2):=\left(\begin{pmatrix}
i & 0 \\
0 & 1
\end{pmatrix},
\begin{pmatrix}
iI_{r-1} & 0 \\
0 & I_{r-1}
\end{pmatrix}\right)\in GL_2(q)\times GL_{2(r-1)}(q).
\end{align*}
Suppose $D \cong Q_8\times Q_8$. Then we may choose $D$ such that the unique block of $\cO HD$ covering $b$ has defect group $D$. Moreover, the image of $D$ in $\Out(H)$ is generated by the non-trivial diagonal automorphism.

We will have reached our desired contradiction once we have proved that no $h\in D\backslash P$ commutes with $P_1$. Since
\begin{align*}
C_{Sp_{2r}(q)}(Z(P_1))=C_{Sp_{2r}(q)}\left(\begin{pmatrix}
-1 & 0 \\
0 & -1
\end{pmatrix},
\begin{pmatrix}
I_{r-1} & 0 \\
0 & I_{r-1}
\end{pmatrix}\right)=Sp_2(q)\times Sp_{2(r-1)}(q),
\end{align*}
it is enough to show there exists no $g\in Sp_2(q)$ such that $g_1g^{-1}\in C_{CSp_2(q)}(P_1)$. This follows from the fact that $C_{CSp_2(q)}(P_1)=Z(CSp_2(q))$ and that conjugation by $g_1$ is not an inner automorphism of $Sp_2(q)$. These two facts can be readily checked.
\end{proof}


\section{Blocks with defect group $Q_8 \times C_{2^n}$ or $Q_8 \times Q_8$}
\label{sec:background}

We begin by gathering together some information on subgroups and automorphism groups of these $2$-groups, easily verified by the reader. Blocks with defect group $Q_8 \times C_{2^n}$ were studied in~\cite{sa13}, where many of their numerical invariants were computed.

\begin{lemma}
\label{Dsubs:Lem}
\begin{enumerate}[(a)]
\item Let $P \cong Q_8 \times C_{2^n}$, where $n \geq 1$. Let $Q \leq P$ with $[P:Q]=2$.

\begin{enumerate}[(i)]
\item If $n=1$, then $Q \cong Q_8$ or $C_4 \times C_2$.

\item If $n \geq 2$, then $Q\cong C_4 \times C_{2^n}$, $Q_8 \times C_{2^{n-1}}$ or $C_4 \rtimes C_{2^n}$.
\end{enumerate}

In particular, every subgroup of $P$ has the form $1$, $C_{2^m}$, $C_2 \times C_{2^m}$, $C_4 \times C_{2^m}$, $Q_8$, $Q_8 \times C_{2^m}$ or $C_4 \rtimes C_{2^m}$ for some $m$, with $m\geq 2$ in the final case.

\item The proper subgroups of $Q_8 \times Q_8$ are isomorphic to the following: $1$, $C_2$, $C_4$, $C_2 \times C_2$, $C_4 \times C_2$, $C_4 \times C_4$, $Q_8$, $Q_8 \times C_2$, $Q_8 \times C_4$, $C_4 \rtimes C_4$ and $C_4 \rtimes Q_8$.
\end{enumerate} 
\end{lemma}

\begin{lemma}
\label{Dautos:Lem}
\begin{enumerate}[(i)]
\item $\Aut(C_{2^n} \times C_{2^m})$ is a $2$-group if $m \neq n$, and is a $\{2,3\}$-group with Sylow $3$-subgroup of order $3$ if $m=n$.

\item $\Aut(Q_8 \times C_{2^n})$ is a $\{2,3\}$-group with Sylow $3$-subgroup of order $3$.

\item $\Aut(Q_8 \times Q_8)$ is a $\{2,3\}$-group with Sylow $3$-subgroup $C_3 \times C_3$.

\item $\Aut(C_4 \rtimes C_{2^n})$ is a $2$-group.

\item $\Aut(C_4 \rtimes Q_8)$ is a $2$-group.
\end{enumerate}
\end{lemma}

The key to our treatment of blocks with defect group $Q_8 \times C_{2^n}$ or $Q_8 \times Q_8$ is that in most cases covered blocks of normal subgroups of index $2$ are nilpotent. This is covered in the following lemma.

\begin{lemma}
\label{cyclicnilpotent}
Let $B$ be a $2$-block of a finite group $G$ with defect groups isomorphic to (i) $C_{2^m} \times C_{2^n}$ for $m \neq n$, (ii) $C_4 \rtimes C_{2^n}$ for $n \geq 2$, or (iii) $C_4 \rtimes Q_8$. Then $B$ is nilpotent.
\end{lemma}

\begin{proof}
Let $D$ be a defect group for $B$.

(i) Since $D$ is abelian, it suffices to observe that $\Aut(D)$ is a $2$-group.

(ii) Since $C_4 \rtimes C_{2^n}$ is metacyclic, by~\cite[Theorem 3.7]{cg12} there is only one saturated fusion system on this $2$-group, and so $B$ must be nilpotent.

(iii) There is only one saturated fusion system on $C_4 \rtimes Q_8$ by~\cite[Table 13.1]{sa14}.
\end{proof}

We summarize the results of~\cite{sa13} that we need here:

\begin{proposition}[\cite{sa13}]
\label{numbersimples}
Let $B$ be a block with defect group $Q_8 \times C_{2^n}$ for some $n$. Then one of the following occurs:

\begin{enumerate}[(i)]
\item $k(B)=2^n\cdot 7$ and $l(B)=3$;
\item $B$ is nilpotent, $k(B)=2^n \cdot 5$ and $l(B)=1$.
\end{enumerate}
\end{proposition}

\begin{proof}
This follows from~\cite[Lemma 2.2, Theorem 2.7]{sa13}.
\end{proof}

The next result will be used frequently without reference throughout the remainder of the article.

\begin{proposition}\label{nilpotent_index_p:prop}
Let $G$ be a finite group and $N \lhd G$ with $G/N$ a $p$-group. Let $b$ be a block of $\cO N$ and $B$ the unique block of $\cO G$ covering $b$. Then $B$ is nilpotent if and only if $b$ is.
\end{proposition}

\begin{proof}
That $B$ nilpotent implies $b$ nilpotent is clear. The other direction is~\cite[Theorem 2]{ca87}.
\end{proof}

\begin{lemma}
\label{odd_index_Q_8xC2^n:lemma}
Let $G$ be a finite group and $N \lhd G$ such that $G/N$ is supersolvable of odd order. Let $B$ be a block of $\cO G$ with defect group $Q_8 \times C_{2^n}$ for some $n$ and let $b$ be a block of $\cO N$ covered by $B$. Suppose that $B$ covers no nilpotent block of any normal subgroup containing $N$. Then $B$ and $b$ are Morita equivalent.
\end{lemma}

\begin{proof}
Note that $B$ and $b$ must share a defect group, and that for any $M \lhd G$ with $N \leq M$ and any block $C$ of $\cO M$ covered by $B$ we have $l(C)=3$ by Proposition \ref{numbersimples}. By considering a chief series between $N$ and $G$ with prime factors it suffices to consider the case that $[G:N]$ is an odd prime $w$. If $b$ is not $G$-stable, then we are done. Suppose that $b$ is $G$-stable. By~\cite[Proposition 2.2]{kkl12} if $G$ acts as inner automorphisms on $b$, then $B$ and $b$ are Morita equivalent. By~\cite[Proposition 2.3]{kkl12} if $G$ does not act as inner automorphisms on $b$, then $B$ is the unique block of $\cO G$ covering $b$. Consider the action of $G$ on the three irreducible Brauer characters of $b$. If $w \geq 5$, then $G$ must fix every such Brauer character, and $l(B)=wl(b)=3w$, contradicting $l(B)=3$. If $w=3$, then either $G$ fixes each irreducible Brauer character, again a contradiction, or $G$ permutes the irreducible Brauer characters of $b$ transitively and $l(B)=1$, contradicting our assumption that $B$ is not nilpotent (using Proposition \ref{numbersimples}).
\end{proof}

\begin{lemma}
\label{coveringC4xC4:lem}
Let $G$ be a finite group and $N \lhd G$ with $[G:N]=2$. Let $b$ be a non-nilpotent block of $\cO N$ with defect groups isomorphic to $C_4 \times C_4$. Then the unique block $B$ of $\cO G$ covering $b$ cannot have defect group $Q_8 \times C_4$.
\end{lemma}

\begin{proof}
We assume the contrary. By Proposition \ref{numbersimples} we have $k(B)\geq 20$ and by~\cite[Theorem 1.1]{ekks14} $k(b)=8$. But by Clifford theory the number of irreducible characters of $G$ lying over irreducible characters of $b$ is at most $16$, and the result follows.
\end{proof}

A block of a finite group $G$ is described as quasiprimitive if every covered block of a normal subgroup of $G$ is $G$-stable. The following is proved in~\cite[Lemma 2.4]{ar19}.

\begin{lemma}
\label{solv_quotient:lemma}
Let $G$ be a finite group and $B$ be a quasiprimitive block of $\cO G$ with defect group $D$. Let $N \lhd G$ such that $G/N$ is solvable and $b$ the unique block of $\cO N$ covered by $B$. Then $ND/N \in \Syl_p(G/N)$.
\end{lemma}

Recall that a $p$-solvable group $G$ has $p$-length one if there are normal subgroups $N, M \lhd G$, with $N \leq M$, such that $N$ and $G/M$ are $p'$-groups and $M/N$ is a $p$-group. The abelian Sylow $2$-subgroup case of the following is well-known, and possibly also the $Q_8$ case, but since we do not know a reference we include a proof. 

\begin{lemma}
\label{two_length:lemma}
Let $G$ be a solvable group with Sylow $2$-subgroups which are abelian or $Q_8$. Then $G$ has $2$-length one. If further $G$ has cyclic $2$-subgroups, then it is $2$-nilpotent.
\end{lemma}

\begin{proof}
We may assume that $O_{2'}(G)=1$, 
so that $O_2(G) \neq 1$ and $C_G(O_2(G)) \leq O_2(G)$. Let $P \in \Syl_2(G)$. If $P$ is abelian, then $P \leq C_G(O_2(G)) \leq O_2(G)$ and we are done. Suppose that $P \cong Q_8$. Since $Z(P)$ is the unique subgroup of $P$ of order two, we have $Z(P) \leq O_2(G)$. If $O_2(G)=Z(P)$ or $O_2(G)=P$, then we are done. If $O_2(G) \cong C_4$, then $G/O_2(G)$ is a $2$-group and again we are done.

In the case that $P$ is cyclic, the fact that $G$ is $2$-nilpotent follows from $\Aut(P)$ being a $2$-group. 
\end{proof}

The following is by now a standard reduction when treating Donovan's conjecture, using Fong reductions and~\cite{kp90}. 
\begin{proposition}
\label{reduce:prop}
Let $G$ be a finite group and let $B$ be a block of $\cO G$ with defect group $D$. There is a finite group $H$ with $[H:O_{p'}(Z(H))] \leq [G:O_{p'}(Z(G))]$ and a block $C$ of $\cO H$ with defect group $P \cong D$ such that $B$ is Morita equivalent to $C$ and the following are satisfied:
\begin{enumerate}
\item[(R1)] $C$ is quasiprimitive;
\item[(R2)] If $N \lhd H$ and $C$ covers a nilpotent block $c$ of $\cO N$, then $N \leq Z(H) O_p(H)$;
\item[(R3)] $O_{p'}(Z(H)) \leq [H,H]$.
\end{enumerate}
\end{proposition}

\begin{proof}
See the first part of the proof of~\cite[Proposition 4.3]{eel19}.
\end{proof}

For the purpose of this article, we call the pair $(H,C)$, where $C$ is a block of $\cO H$, \emph{reduced} if it satisfies conditions (R1), (R2) and (R3) of Proposition \ref{reduce:prop}. If the group is clear, then we just say $C$ is reduced. The property of being reduced is very restrictive in our situation:

\begin{proposition}
\label{structureQ8xC2n:prop}
Let $(G,B)$ be a reduced pair, where $B$ has defect group $D$.
\begin{enumerate}[(a)]
\item If $D \cong Q_8 \times C_{2^n}$, then either:
\begin{enumerate}[(i)]
\item there is $N \lhd G$ such that $G=ND$ and $D \cap N \cong Q_8$ (so $G/N$ is cyclic); or
\item there is $N \lhd G$, a quasisimple group that is a quotient of a group of classical Lie type other than $\mathbf{A}$ or ${}^2\mathbf{A}$, such that $G/N$ is $2$-nilpotent with cyclic Sylow $2$-subgroup $ND/N$ and $O_{2'}(G/N)$ is supersolvable. The unique block of $\cO N$ covered by $B$ is not quasi-isolated and $D \cap N$ has a subgroup isomorphic to $Q_8$.
\end{enumerate}
\item If $D \cong Q_8 \times Q_8$, then either:
  \begin{enumerate}[(i)] 
  \item there are normal subgroups $N \lhd H \lhd G$ such that $H=ND$, $D \cap N \cong Q_8$ or $Q_8 \times C_2$, and $[G:H]$ is odd; or
  \item there are commuting, normal subgroups $N_1, N_2 \lhd G$ such that $N_1 \cap N_2 \leq Z(G)$, $D \cap N_1$, $D \cap N_2 \cong Q_8$ or $Q_8 \times C_2$, and $[G:N_1N_2]$ is odd.
  \end{enumerate} 
  
\end{enumerate}
\end{proposition}

\begin{proof}
Let $G$ be a finite group and $B$ a block of $\cO G$ with defect group $D \cong Q_8 \times C_{2^n}$ or $Q_8 \times Q_8$ for some $n \geq 1$, and suppose that $(G,B)$ is reduced. Write $E(G)$ for the layer of $G$, the central product of all of the components of $G$, and $F^*(G)$ for the generalized Fitting subgroup (see~\cite{asc00}). Then $F^*(G)=E(G)Z(G)O_2(G)$ by our assumption, and $C_G(F^*(G)) \leq F^*(G)$ because of general properties of $F^*(G)$. Write $E(G)$ for the central product $L_1 \circ \cdots \circ L_t$, where the $L_i$ are the components of $G$. Write $B_E$ for the unique block of $\cO E(G)$ covered by $B$. Then (using the fact that $B_E$ is $G$-stable) $D_E:=D \cap E(G)$ is a defect group for $B_E$ and $D_i:=D \cap L_i$ is a defect group for the unique block $b_i$ of $\cO L_i$ covered by $B_E$. By Lemma \ref{Dsubs:Lem} $O_2(G)$, $D_E$ and each $D_i$ is of the form $1$, $C_{2^m}$, $C_2 \times C_2$, $C_4 \times C_{2^m}$, $Q_8$, $Q_8 \times C_{2^m}$, $Q_8 \times Q_8$, $C_4 \rtimes C_{2^m}$ or $C_4 \rtimes Q_8$ for some $m$ (with $m \geq 2$ in the eighth case). Write $Z_E:=O_2(E(G))$. Note that $Z_E$ is central in $E(G)$. Now the unique block $\overline{B}_E$ of $\cO (E(G)/Z_E)$ corresponding to $B_E$ has defect group $\overline{D}_E:=D_E/Z_E \cong D_1Z_E/Z_E \times \cdots \times D_tZ_E/Z_E$ and has $2$-rank at most four (we are using that $\overline{B}_E$ corresponds to a block of a direct product of quasisimple groups of which $E(G)/Z_E$ is a quotient by a $2'$-group).

Suppose $t \geq 3$, or $t=2$ in case (a). Then at least one $D_iZ_E/Z_E$ is a cyclic $2$-group, so that the unique block $\overline{b}_i$ of $L_iZ_E/Z_E$ corresponding to $b_i$ is nilpotent. Let $M$ be the product of those $L_i$ such that $\overline{b}_i$ is nilpotent. Then $M$ must be a normal subgroup (since $G$ permutes the components). Let $B_M$ be the unique block of $\cO M$ covered by $B$ and $\overline{B}_M$ be the unique block of $MZ_E/Z_E$ corresponding to $B_M$. Now $\overline{B}_M$ is isomorphic to a product of nilpotent blocks and so is itself nilpotent ($\overline{B}_M$ corresponds to a block of a direct product of quasisimple groups of which $MZ_E/Z_E$ is a quotient by a $2'$-group). By~\cite{wa94} $B_M$ is then nilpotent, a contradiction to our assumption that $B$ is reduced. Hence $t \leq 2$, with equality only if $D \cong Q_8 \times Q_8$. It follows from Schreier's conjecture and Lemmas~\ref{Dsubs:Lem} and~\ref{Dautos:Lem} that $G/F^*(G)$ has a solvable subgroup of index at most two (namely the subgroup stabilizing the components if $t \geq 1$, and $G/F^*(G)$ itself if there are no components) and so $G/F^*(G)$ is itself solvable. Similarly $G/E(G)$ is solvable. By Lemma \ref{solv_quotient:lemma} $DF^*(G)/F^*(G)$ is a Sylow $2$-subgroup of $G/F^*(G)$ and $DE(G)/E(G)$ is a Sylow $2$-subgroup of $G/E(G)$, facts we will use frequently.

Suppose that $t=2$, in which case $D \cong Q_8 \times Q_8$. Note also that each component is normal in $G$ since $G/E(G)$ is of odd order (i.e., there is no involution permuting the two components). By considering all of the possible expressions of $Q_8 \times Q_8$ as a central product of two groups in Lemma \ref{Dsubs:Lem}, we must have (without loss of generality) $D_1$ and $D_2 \cong Q_8$ or $Q_8 \times C_2$, otherwise $B$ covers a block of a component with cyclic defect group, which forces a contradiction as in the previous paragraph. Note that $O_2(G) \leq Z(G)$ in this case. We are now in case (b)(ii) of the statement, with $N_i=L_i$.

Suppose that $t=0$. Then $F^*(G)=O_2(G)Z(G)$. Since $C_G(F^*(G)) \leq F^*(G)$, we have $G/F^*(G) \leq \Out(O_2(G))$. In particular $O_2(G)$ is a self-centralizing subgroup of $D$. Note that if $D = O_2(G)$, then $G\cong D\rtimes E$ for some $2'$-group $E$. It follows from Lemma~\ref{Dautos:Lem} that $G/C_E(D)$ is a subgroup of $SL_2(3) \times C_{2^n}$ or $SL_2(3)\times SL_2(3)$ and we must be in case (a)(i) or (b)(i) respectively. Suppose $O_2(G)$ is a proper subgroup of $D$. Note that, by Lemma \ref{Dautos:Lem}, all self-centralizing subgroups of $D$ have solvable automorphism group.
Suppose first that $D \cong Q_8 \times C_{2^n}$. By Lemma \ref{Dsubs:Lem} the only proper self-centralizing normal subgroup of $D$ is $C_4 \times C_{2^n}$ and we need only consider the case $O_2(G) \cong C_4 \times C_4$ as in the other cases the automorphism group is a $2$-group. In this case we must have that $G/O_2(G) \cong S_3$ and $B$ covers a non-nilpotent block of $O^2(G)$. However this cannot happen by Lemma \ref{coveringC4xC4:lem}. Now suppose that $D \cong Q_8 \times Q_8$. By Lemma \ref{Dsubs:Lem} the proper self-centralizing normal subgroups of $D$ are $C_4 \times C_4$, $Q_8 \times C_4$ and $C_4 \rtimes Q_8$, of which we need only consider $C_4 \times C_4$ and $Q_8 \times C_4$ since $C_4 \rtimes Q_8$ has automorphism group a $2$-group. If $O_2(G) \cong C_4 \times C_4$, then $G/F^*(G)$ has order $12$ and so has a non-trivial normal $2$-subgroup, a contradiction. If $O_2(G) \cong Q_8 \times C_4$, then $G/F^*(G) \cong S_3$. This forces $G \cong SL_2(3) \times Q_8$ and we are in case (a)(i).



Now suppose that $t=1$, so $F^*(G)=O_2(G)Z(G)L_1$.
Since, by Lemma \ref{cyclicnilpotent}, every block with defect group $C_4 \rtimes Q_8$, $C_4 \rtimes C_{2^m}$ or $C_4 \times C_{2^s}$ is nilpotent for $m \geq 2$ and $s \neq 2$, we may assume, by Lemma \ref{Dsubs:Lem}, that $D_1 \cong Q_8 \times Q_8$, $Q_8 \times C_{2^m}$, $Q_8$, $C_4 \times C_4$ or $C_2 \times C_2$ for some $m \geq 1$. We treat each of these cases in turn.

By Proposition \ref{none_in_quasisimples:prop} we cannot have $D_1 \cong Q_8 \times Q_8$. 

Suppose that $D_1 \cong Q_8 \times C_{2^m}$ for $m \geq 1$. Consider first the case $D \cong Q_8 \times C_{2^n}$. By Proposition \ref{none_in_quasisimples:prop} $L_1$ is of classical type other than $\mathbf{A}$ or ${}^2\mathbf{A}$ as in case (a)(ii) of the statement with $N=L_1$. That $G/N$ has cyclic Sylow $2$-subgroup $ND/N$ is immediate, and so by Lemma \ref{two_length:lemma} $G/N$ is $2$-nilpotent. Considering the outer automorphism groups of such quasisimple groups we have that $O_{2'}(G/L_1)$ is supersolvable as required. As a note to this last calculation, observe that unless $L_1$ has type $\mathbf{D}_4$ the only odd order elements of the outer automorphism group are field automorphisms, so that $O_{2'}(G/L_1)$ is cyclic. In the case of type $\mathbf{D}_4$ we may either analyse~\cite[Theorem 1.1]{an20} a little more deeply and observe that this case does not after all occur, or observe that a Hall $2'$-subgroup of $\Out(L_1)$ is a subgroup of $C_3 \times Y$ where $Y$ is cyclic, implying that $O_{2'}(G/L_1)$ is supersolvable.

Now suppose that $D_1 \cong Q_8 \times C_{2^m}$ for $m \geq 1$ and $D \cong Q_8 \times Q_8$. Then $1 \leq m \leq 2$. The case $m=2$ is ruled out by Proposition \ref{none_in_quasisimples:prop} and Lemma \ref{noQ8xQ8ext:Lemma}. If $m=1$, then we are in case (b)(i).

Suppose that $D_1 \cong Q_8$. Then since $G/L_1$ is a solvable group which has cyclic or $Q_8$ Sylow $2$-subgroup, by Lemma \ref{two_length:lemma}, it has $2$-length one. Let $N$ be the preimage in $G$ of $O_{2'}(G/L_1)$. Then the unique block $B_N$ of $\cO N$ covered by $B$ has defect group $Q_8$.  If $D/D_1$ is cyclic, then by Lemma \ref{two_length:lemma} $G/L_1$ is $2$-nilpotent and we are in case (a)(i) of the statement. If $D/D_1 \cong Q_8$, then let $H = O^{2'}(G)$ and we are in case (b)(i) of the statement.



Finally we claim that we cannot have $D_1 \cong C_4 \times C_4$ or $C_2 \times C_2$. Suppose that we do and so $D/D_1$ is an abelian. Let $M$ be the preimage in $G$ of $O_{2'}(G/L_1)$, and let $B_M$ be the unique block of $\cO M$ covered by $B$. Then $B_M$ has defect group $D_1$. Now there is a subgroup $M_1$ of $G$ containing $M$ with $[M_1:M]=2$. The unique block $B_{M_1}$ of $\cO M_1$ covering $B_M$ has defect group $Q_8 \times C_4$ or $C_4\rtimes Q_8$ (if $D_1 \cong C_4 \times C_4$) or $C_4 \times C_2$ (if $D_1 \cong C_2 \times C_2$), since it is a subgroup of $D$. In the $Q_8\times C_4$ case we obtain a contradiction by Lemma \ref{coveringC4xC4:lem}. In the $C_4\rtimes Q_8$ and $C_4\times C_2$ cases by Lemma \ref{cyclicnilpotent} $B_M$, and so by Proposition \ref{nilpotent_index_p:prop} $b_i$, must be nilpotent, a contradiction.

\end{proof}

\begin{corollary}
\label{sfQ8xC:cor}
Let $G$ be a finite group and $B$ a block of $\cO G$ with defect group $D$ isomorphic to $Q_8 \times C_{2^n}$ for $n \geq 0$. Then $\SOF(B) \leq (|D|^2)!$ and ${\rm c}(B) \leq 2^{n+3}$.
\end{corollary}

\begin{proof}
Morita equivalence preserves the Cartan invariants, and by~\cite[Proposition 3.12]{el19} Morita equivalence of $\cO$-blocks preserves $\SOF$, so by Proposition \ref{reduce:prop} it suffices to consider reduced blocks. Apply Proposition \ref{structureQ8xC2n:prop}.

First let $(G,B)$ be a reduced pair satisfying condition (a)(ii) of Proposition \ref{structureQ8xC2n:prop}. Let $B'$ be the unique block of $\cO O^2(G)$ covered by $B$. By Lemma~\ref{Cartan_reduce:lemma}(i) ${\rm c}(B) \leq [G:O^2(G)]{\rm c}(B')$ and by Theorem \ref{sfreduce:theorem} $\SOF(B) \leq \SOF(B')$. By considering the outer automorphism groups of the classical groups of Lie type (see for example~\cite[Table 5]{atlas}) we see that every normal subgroup of $O^2(G)$ is also normal in $G$, so $B'$ is quasiprimitive. Hence we see that $B'$ is also a reduced block satisfying (a)(ii). We may now assume that $O^2(G)=G$. Now let $B_N$ be the unique block of the quasisimple group $N$ covered by $B$. By Lemma \ref{odd_index_Q_8xC2^n:lemma} $B$ is Morita equivalent to $B_N$, and note that they share a defect group. Hence we may assume that $G=N$ since Morita equivalence preserves both of these invariants (see~\cite[Proposition 3.12]{el19} for the latter).

We make use of~\cite{bdr17}, to which we refer for notation. Assume for the moment that $G$ is a group of Lie type, i.e, the centre is largest possible. We note that $Z(G)$ is a $2$-group for the groups we are considering. Here the identity element is the only quasi-isolated element (see for example~\cite[Table 2]{bo05}) and so the principal block is the only quasi-isolated block.  However, for groups of these types the principal block of $G$ (or that of any quotient of $G$) cannot have the given defect groups, so we may assume our block is not quasi-isolated. We may now apply the Bonnaf\'e-Dat-Rouquier correspondence~\cite[Theorem 7.7]{bdr17}, so that $B$ is Morita equivalent to a block $C_1$ (with isomorphic defect group) of a proper subgroup $H_1$ of $G$. We note that the well-known error in~\cite{bdr17} does not apply in our situation, since we are working with $2$-blocks and the centre of $G$ is a $2$-group. If $G$ is not a group of Lie type (i.e., the centre is not largest possible), then we note that by, for example,~\cite[Proposition 4.1]{ekks14} the Bonnaf\'e-Dat-Rouquier correspondence induces a Morita equivalence modulo central $2$-subgroups and we may apply the same argument.

Applying Proposition \ref{reduce:prop} to $C_1$, it is Morita equivalent to a block $C_2$ in a reduced pair $(H_2,C_2)$ with isomorphic defect groups, where $[H_2:O_{2'}(Z(H_2))] \leq [H_1:O_{2'}(Z(H_1))]<[G:O_{2'}(Z(G))]$. We have ${\rm c}(B) = {\rm c}(C_1)$ and $\SOF(B)=\SOF(C_1)$. Now apply Proposition \ref{structureQ8xC2n:prop} to $(H_2,C_2)$. Either we are in case (a)(i) of Proposition \ref{structureQ8xC2n:prop} or we may repeat the above argument. Since the index of the $2'$-part of the centre strictly decreases each time we apply the Bonnaf\'e-Dat-Rouquier correspondence, repetition of this process must eventually end in case (a)(i) of Proposition \ref{structureQ8xC2n:prop}. 

Now let $(G,B)$ be a reduced pair satisfying condition (a)(i), so there is $N \lhd G$ such that $G=ND$ and $G \cap N \cong Q_8$. Let $b$ be the unique block of $\cO N$ covered by $B$, noting that this has defect group $Q_8$. By~\cite{ei16} there is a unique Morita equivalence class of blocks with defect group $Q_8$ and a given Cartan matrix, so $\mfO(b)=1$. By Theorem \ref{sfreduce:theorem} and~\cite[Corollary 3.11]{el19} $\SOF(B) \leq \SOF(b) \leq |Q_8|^2!$. Considering Cartan invariants, ${\rm c}(b) \leq 8$ and so, by Lemma~\ref{Cartan_reduce:lemma}(i), ${\rm c}(B) \leq 2^{n+3}$.
\end{proof}



We remark that we cannot at present so easily obtain a similar bound on the strong $\cO$-Frobenius number for blocks with defect group $Q_8 \times Q_8$, since we do not know how this invariant behaves with respect to normal subgroups of $p'$-index.

We further remark that in order to bound only the strong Frobenius number of a quasisimple group we could have used~\cite{fk18}.


\section{Donovan's conjecture for blocks with defect group $Q_8 \times C_{2^n}$ or $Q_8 \times Q_8$}
\label{sec:Donovan_cases}

We are now in a position to verify Donovan's conjecture for $Q_8 \times C_{2^n}$ and $Q_8 \times Q_8$.
\newline
\newline
{\sc Proof of Theorem}~\ref{Q8Ctheorem}.


By Proposition~\ref{reduce:prop}, in verifying Donovan's conjecture it suffices to consider reduced blocks.

Consider first $Q_8\times C_{2^n}$. From~\cite[Corollary 3.11]{eel19} we need only bound strong $\cO$-Frobenius numbers and Cartan invariants for reduced blocks with defect groups isomorphic to $Q_8\times C_{2^n}$, hence the result follows in this case by Corollary \ref{sfQ8xC:cor}.

Now consider blocks with defect groups $Q_8 \times Q_8$. By Proposition~\ref{structureQ8xC2n:prop} either: (i) there are normal subgroups $N \lhd H \lhd G$ with $H=ND$ and $D \cap N \cong Q_8$ or $Q_8 \times C_2$, and $[G:H]$ odd; or (ii) there are commuting $N_1, N_2 \lhd G$ with $N_1 \cap N_2 \leq Z(G)$ such that $D \cap N_1$, $D \cap N_2 \cong Q_8$ or $Q_8 \times C_2$, and $[G:N_1N_2]$ is odd.

By~\cite[Corollary 4.18]{ei18} it suffices to show that there are only finitely many possibilities for the Morita equivalence class of the unique block of $\cO H$ or $\cO N_1 N_2$ covered by $B$. Hence we may assume that $G=H$ in case (i) and $G=N_1N_2$ in case (ii).

As above, from~\cite[Corollary 3.11]{eel19} we need only bound strong $\cO$-Frobenius numbers and Cartan invariants for  such blocks. 

In case (i) we have $\SOF(B) \leq 16^2!$ and ${\rm c}(B) \leq 2^6$ using arguments as in Corollary \ref{sfQ8xC:cor}, noting that $D\cong Q_8 \times Q_8$ with $D\cap N\cong Q_8$ or $Q_8 \times C_2$ satisfies the conditions of Theorem \ref{sfreduce:theorem}. Suppose that we are in case (ii). Note that $G \cong (N_1 \times N_2)/W$ for some group $W \leq Z(N_1 \times N_2)$. Now $B$ corresponds to a block $A$ of $N_1 \times N_2$ with $O_{2'}(W)$ in its kernel and defect group $Q_8 \times Q_8$, $Q_8 \times Q_8 \times C_2$ or $Q_8\times Q_8\times C_2\times C_2$. Write $A_i$ for the block of $N_i$ covered by $A$, with defect group $Q_8$ or $Q_8 \times C_2$. As above we have $\SOF(A_i)\leq 16^2!$, for $i=1,2$. Hence $\SOF(A) \leq (16^2!)^2$. By~\cite[Proposition 3.17]{el19} we have $\SOF(B) \leq \SOF(A)$. Finally we have ${\rm c}(A_i) \leq 16$, for $i=1,2$, so by Lemma~\ref{Cartan_reduce:lemma}(ii) ${\rm c}(B) \leq {\rm c}(A) \leq 16^2$.
\hfill $\Box$

\medskip

The authors are aware of only three Morita equivalence classes of blocks with defect group $Q_8 \times C_{2^n}$ for each $n \geq 0$, namely those with representatives the principal blocks of $\cO (Q_8 \times C_{2^n})$, $\cO (SL_2(3) \times C_{2^n})$ and $\cO(SL_2(5) \times C_{2^n})$. Similarly the known Morita equivalence classes of blocks with defect group $Q_8 \times Q_8$ have representatives the principal blocks of $\cO(Q_8 \times Q_8)$, $\cO(Q_8 \times SL_2(3))$, $\cO((Q_8 \times Q_8)\rtimes C_3)$ (SmallGroup(192,1022), where $C_3$ acts on $(Q_8 \times Q_8)/Z(Q_8 \times Q_8)$ with only one fixed point), $\cO(Q_8 \times SL_2(5))$, $\cO(SL_2(3) \times SL_2(3))$, $\cO(SL_2(3) \times SL_2(5))$ and $\cO(SL_2(5) \times SL_2(5))$, and a non-principal block of $(Q_8 \times Q_8) \rtimes 3_+^{1+2}$, where the centre of $3_+^{1+2}$ acts trivially.


\bigskip

\begin{center} ACKNOWLEDGEMENTS \end{center}

We thank Mandi Schaeffer-Fry, Lucas Ruhstorfer and Jay Taylor for some useful conversations. This research was supported by EPSRC grant no. EP/T004606/1.

\end{document}